\documentclass[11pt]{article}

\usepackage{amsthm,amsfonts,amssymb,amsmath,color}

\numberwithin{equation}{section}

\renewcommand\d{\partial}
\renewcommand\a{\alpha}
\renewcommand\b{\beta}

\newcommand\N{\mathbb N}\newcommand\Z{\mathbb Z}

\def\g{\gamma}
\def\de{\delta}

\def\O{\Omega}
\def\th{\theta}

\def\vp{\varphi}
\def\epsilon{\varepsilon}
\def\e{\varepsilon}

\newcommand\br{\begin{rem}}
\newcommand\er{\end{rem}}
\newcommand\bp{\begin{pmatrix}}
\newcommand\ep{\end{pmatrix}}
\newcommand\be{\begin{equation}}
\newcommand\ee{\end{equation}}
\newcommand\ba{\begin{equation}\begin{aligned}}
\newcommand\ea{\end{aligned}\end{equation}}

\newcommand\nn{\nonumber}

\newcommand{\CalB}{\mathcal{B}}

\newcommand{\II}{{\mathbb I}}

\newcommand{\SSS}{{\mathbb S}}

\newcommand{\supp}{{\rm supp }}

\newcommand{\vv}{{\mathbf v}}

\newcommand{\vr}{\varrho}

\newcommand{\vu}{\vc{u}}
\newcommand{\vf}{\vc{f}}
\newcommand{\vg}{\vc{g}}
\newcommand{\vc}[1]{{\bf #1}}
\newcommand{\Div}{{\rm div\,}}
\newcommand{\Grad}{\nabla_x}

\newcommand{\dx}{{\rm d} {x}}

\newcommand{\intO}[1]{\int_{\O} #1 \ \dx}
\newcommand{\intOe}[1]{\int_{\O_\e} #1 \ \dx}

\newcommand{\dive}{{\rm div\,}}

\newtheorem{defi}{Definition}[section]
\newtheorem{theorem}[defi]{Theorem}

\newtheorem{lemma}[defi]{Lemma}

\newtheorem{remark}[defi]{Remark}

\numberwithin{equation}{section}
\hoffset - 1 cm

\begin{document}

\title{The inverse of the divergence operator on perforated domains with applications
to homogenization problems for the compressible Navier-Stokes system}

\author{Lars Diening\footnote{Institute of Mathematics, Universit\"at Osnabr\"uck, Albrechtstr. 28a, 49076 Osnabr\"uck, Germany, {\tt lars.diening@uni-osnabrueck.de}} \and  Eduard Feireisl\footnote{Institute of Mathematics of the Academy of Sciences of the Czech Republic, Zitn\' a 25, 115 67 Praha 1, Czech Republic, {\tt feireisl@math.cas.cz}} \and  Yong Lu\footnote{Mathematical Institute, Faculty of Mathematics and Physics, Charles University, Sokolovsk\'a 83, 186 75 Praha, Czech Republic, {\tt luyong@karlin.mff.cuni.cz}}}

\date{}

\maketitle

\begin{abstract}

We study the inverse of the divergence operator on a domain $\Omega \subset R^3$ perforated by a system of tiny holes. We show that such inverse can be constructed on the Lebesgue space $L^p(\Omega)$ for any $1< p < 3$, with a norm independent of perforation, provided the holes are suitably small and their mutual distance suitably large. Applications are given to problems arising in homogenization of steady compressible fluid flows.

\end{abstract}


\renewcommand{\refname}{References}


\section{Introduction}

Homogenization in fluid mechanics gives rise to system of partial differential equations considered on physical domains perforated by
a large number of tiny holes. The typical diameter and mutual distance of these holes play a crucial role in the asymptotic behavior of fluid flows
in the regime where the number of holes tends to infinity and their size tends to zero.

 Viscous fluid flows passing an array of fixed solid obstacles is a situation frequently occurring in
applications. A priori, the Navier-Stokes equations with a no-slip boundary condition on the obstacles are
believed to be the correct model. With an increasing number of holes, the fluid flow approaches an
effective state governed by certain "homogenized" equations which are homogeneous in form (without obstacles).  We refer to \cite{book-hom} for a number of real world applications.

 The problem is relatively well understood in the framework of \emph{stationary, viscous} fluid flows represented by the the standard Stokes and/or
Navier-Stokes system of equations. Allaire \cite{ALL-NS1}, \cite{ALL-NS2} (see also earlier results by Tartar \cite{Tartar1}) identified three different scenarios for the case of periodically distributed
holes:
\begin{itemize}
\item  the supercritical size of holes for which the asymptotic limit is Darcy's law;
\item  the critical size giving rise to Brinkman's law;
\item the subcritical size of holes has no influence on the motion in the asymptotic limit - the limit problem coincides with the original one.

\end{itemize}

Related results for the evolutionary (time-dependent) incompressible Navire-Stokes system were obtained by Mikeli\'{c} \cite{Mik} and, more recently, in
\cite{FeNaNe}.

Considerably less is known in the case of \emph{compressible} fluids. Masmoudi \cite{Mas-Hom} identified rigorously the porous medium equation (Darcy's law)  as a homogenization
limit for the evolutionary barotropic (compressible) Navier-Stokes system in the case where the diameter of the holes is comparable to their mutual distance, which is a subcase of the supercritical case, similar results for the full Navier-Stokes-Fourier system were obtained in \cite{FNT-Hom}.

In \cite{FL1}, we considered the compressible (isentropic) stationary Navier-Stokes system in the subcritical regime, where the spatial domain is perforated by a periodic lattice of holes of subcritical size. Similarly to the incompressible case, we showed that the motion is not affected by the obstacles and the limit problem coincides with the original one. The result was conditioned by two basic hypotheses:
\begin{itemize}
\item
the isentropic pressure-density state relation
\begin{equation}
\label{H1}
p(\vr) = a \vr^\gamma, \ a > 0, \ \gamma \geq 3;
\end{equation}
\item
periodic distribution of the holes.

\end{itemize}

Note that hypothesis (\ref{H1}) was also used by Masmoudi \cite{Mas-Hom} in the evolutionary case. In the stationary regime considered in \cite{FL1}, the
assumption (\ref{H1}) plays a crucial role as it renders the pressure square-integrable. Accordingly, the well developed Hilbertian $L^2-$theory can be used to handle the problem, in particular, the restriction operator introduced by Tartar \cite{Tartar1} can be used in a compatible way to construct the inverse of the
divergence - the so-called Bogovskii's operator (see Bogovskii \cite{bog}, Galdi \cite[Chapter 3]{Galdi}).

Our goal in the present paper is to extend the results of \cite{FL1}  to the case:

\begin{itemize}
\item
the isentropic pressure-density state relation with lower adiabatic number
\begin{equation}
\label{H2}
p(\vr) = a \vr^\gamma, \ a > 0, \ \gamma > 2;
\end{equation}
\item general distribution of the holes, only conditions on the diameter, shape, and mutual distance prescribed.
\end{itemize}

While considering a general distribution of holes represents only an incremental improvement with respect to \cite{FL1}, the seemingly easier step from (\ref{H1}) to (\ref{H2})
requires more effort. The reason is that the pressure $p$ is no longer (known to be) square integrable, and, consequently, the $L^2-$theory cannot be used
in order to obtain the necessary uniform bounds on the solutions. In particular, the inverse of the divergence operator used in \cite[Section 2.1]{FL1},
based on the standard Bogovskii's construction acting between the spaces $L^2_0(\O_\e)$ and $W_0^{1,2}(\O_\e; R^3)$ combined with Tartar's restriction
operator, is no longer applicable and must be replaced by its $L^p-$analogue for a general $1 < p < 3$. The construction of such an operator may be seen
as the main novelty of the present paper in comparison with \cite{FL1} and may be of independent interest.

The paper is organized as follows. In Section \ref{P}, we collect the necessary preliminary materials, formulate the problem, and state our main result.
Section \ref{sec:bog-proof} is the heart of the paper. Here we construct the inverse of the divergence -- a Bogovskii's type operator -- enjoying the desired properties. The uniform estimates obtained via this operator are then used in Section \ref{sec:pf-hom} to identify the asymptotic limit for the Navier-Stokes system in perforated domains.

\section{Preliminaries, problem formulation, main result}
\label{P}

In what follows, we denote by $W^{-1,q}(\O)$ the dual space to the the Sobolev spaces $W_0^{1,q'}(\O)$, where
\[
\frac{1}{q}+ \frac{1}{q'}=1,
\]
with the standard norm
\be\label{def-W-1p}
\|u\|_{W^{-1,q}(\O)}:=\sup_{\phi \in C_c^\infty(\O),\,\|\phi\|_{W^{1,q'}_0} \leq 1} \left|\int_{\O} u \phi \ {\rm d}x \right|.
\ee
The symbol $L^{q}_0(\Omega)$ denotes the space of functions in $L^q(\Omega)$ with zero integral mean:
\be\label{Lp0}
L^q_0(\Omega):=\left\{f\in L^q(\Omega): \, \int_{\O} f\,\dx=0\right\}.
\ee

\subsection{Perforated domain}\label{sec:domain}

Consider a bounded domain $\Omega \subset R^3$ of class $C^2$. We introduce a family of \emph{perforated domains}
$\{ \Omega_\e \}_{\e > 0}$,
\be\label{domain}
\Omega_\e = \Omega \setminus \bigcup_{n=1}^{N(\e)} T_{\e, n},
\ee
where the sets $T_{\e,n}$ represent \emph{holes} or \emph{obstacles}.  We suppose the following property concerning the distribution of the holes:
 \be\label{dis-holes}
T_{\e,n} = x_{\e,n} + \e^\a T_n \subset B(x_{\e,n}, \de_0 \e^\alpha ) \subset B(x_{\e,n}, \delta_1 \e) \subset B(x_{\e,n}, \de_2\e)\subset \O,
\ee
where for each $n$,  $T_n\subset R^3$ is a simply connected bounded domain of class $C^2$ and is independent of $\e$, $B(x,r)$ denotes the open ball centered at $x$ with radius $r$ in $R^3$, $\de_0, \ \de_1, \ \de_2$ are positive constants independent of $\e$ and there holds $\de_1<\de_2$.  Moreover, we suppose balls (control volumes) in $\{B(x_{\e,n}, \de_2\e)\}_{n\in \N}$ are pairwise disjoint.

Compared to the assumption on the distribution of holes in \cite{FL1}, here we do not assume the periodicity of the distribution, and we do not assume the uniform shape of the holes.

The diameter of each $T_{\e,n}$ is of order
$O(\e^\alpha)$ and their mutual distance  is $O(\e)$, while their total number $N(\e)$ can be estimated as
\be\label{number-holes}
N_{\e} \leq \frac{3}{4 \pi} \frac{|\Omega|}{\e^3}.\nn
\ee

\subsection{Stationary Navier-Stokes equations}

For the fluid density $\vr = \vr(x)$ and the velocity field $\vu = \vu(x)$,
we consider the stationary (compressible) \emph{Navier-Stokes system}
\be\label{i1}
\dive (\vr \vu) = 0,
\ee
\be\label{i2}
 \dive (\vr \vu \otimes \vu) +\nabla p(\vr) =
\dive \SSS(\nabla \vu)+\vr \vf+\vg,
\ee
\be\label{i3}
\SSS (\nabla \vu) = \mu \left( \nabla \vu + \nabla^t \vu - \frac{2}{3} (\dive \vu) \II \right) + \eta (\dive \vu) \II,\ \mu > 0,\ \eta \geq 0,
\ee
in the spatial domain $\O_\e$, supplemented with the standard no-slip boundary condition
\be\label{i4}
\vu=0\quad \mbox{on}\ \partial\Omega_\e.
\ee

The symbol $\SSS(\nabla \vu)$
stands for the Newtonian viscous stress tensor with constant viscosity coefficients $\mu, \eta$. For the sake of simplicity, we focus on the
\emph{isentropic} pressure-density state equation
\be\label{pp1}
p(\vr)=a\vr^\gamma,\ a > 0,
\ee
with the adiabatic exponent $\g $, the value of which will be specified below.

The motion is driven by the volume force $\vf$ and nonvolume force $\vg$, defined on the whole domain $\Omega$ and independent of $\e$,
that are supposed, again for the sake of simplicity, to be uniformly bounded,
\be\label{ex-force}
\|\vf\|_{L^\infty(\O; R^3)}+\|\vg\|_{L^\infty(\O; R^3)}\leq C<\infty.
\ee
Here and hereafter, the symbol $C$ is used to denote a generic constant that may vary from line to line but it is independent of the parameters of the problem, in particular of $\e$.

Finally, in agreement with its physical interpretation,
the density $\vr$ is non-negative and we fix the total mass of the fluid to be
\be\label{mass}
M = \int_{\Omega_\e} \vr \ \dx  > 0.
\ee

For physical background to these equations and conditions, we refer to Sections 1.2.3, 1.2.4, and 1.2.6.

\subsection{Weak solutions}

We recall the definition of finite energy weak solutions to \eqref{i1}-\eqref{i4}, see e.g. \cite[Definition 4.1]{N-book}.

\begin{defi}\label{def-weaksol}

A couple of functions $[\vr, \vu]$ is said to be a \emph{finite energy weak solution} of the Navier-Stokes system \eqref{i1}-\eqref{i3} supplemented with the conditions \eqref{i4}-\eqref{mass} in $\O_\e$ provided:
\begin{itemize}

\item $\vr \geq 0~ \mbox{ a.e. in}~ \O_\e,$ and
\be\label{wk-sl1}
  \int_{\O_\e}\vr \ \dx = M,\ \vr \in  L^{\beta(\gamma)}(\O_\e) \ \mbox{for some} ~\beta(\gamma)  > \gamma, \ \vu \in  W_0^{1,2}(\O_\e; R^3);
\ee
\item for any test functions   $\psi \in C^\infty( \overline{\O}_\e)$   and $\varphi \in C_c^\infty( {\O_\e}; R^3)$:
\be\label{wk-sl2}
\intOe{ \vr \vu \cdot \nabla \psi  }  =0,
\ee
\be\label{wk-sl3}
 \intOe{ \vr \vu \otimes \vu : \nabla \varphi+ p(\vr) \dive \varphi-\SSS(\nabla \vu) : \nabla \varphi +(\vr \vf+\vg)\cdot \varphi}=0;
\ee

 \item the \emph{energy inequality}
\be\label{wk-sl4}
\intOe{ \SSS(\nabla \vu) : \nabla \vu } \leq \intOe{ (\vr\vf+\vg)\cdot \vu }
\ee
holds.
\end{itemize}

Moreover, a finite energy weak solution $[\vr, \vu]$ is said to be a \emph{renormalized weak solution} if
\be\label{ren}
\int_{R^3}{ b(\vr)\vu \cdot \Grad \psi + \Big( b( \vr) - b'(\vr) \vr \Big) \dive \vu\, \psi }\ \dx = 0
\ee
for any $\psi \in  C_c^\infty(R^3)$, where $[\vr, \vu]$ were extended to be zero outside $\O_\e$,
and any $b\in C^0([0,\infty))\cap C^1((0,\infty))$ such that
\be\label{b-pt1}
b'(s)\leq c\, s^{-\lambda_0} \ \mbox{for}\  s\in (0,1],\quad \ b'(s)\leq c\,s^{\lambda_1} \ \mbox{for}\  s\in [1,\infty),
\ee
with
\be\label{b-pt2}
c>0,\quad \lambda_0<1, \quad -1<\lambda_1 \leq \frac{\beta(\gamma)}{2}-1.
\ee

\end{defi}

\begin{remark}\label{rem:hom2}
By DiPerna-Lions' transport theory (see {\rm \cite[Section II.3]{DiP-L}} and the modification in {\rm \cite[Lemma 3.3]{N-book}}), for any $r\in L^\beta(\Omega),~\beta\geq 2,\ \vv\in W^{1,2}_0(\Omega)$, where $\Omega\subset  R^3$ is a bounded domain of class $C^2$, such that
\begin{equation} \label{rrr}
\dive (r\vv)=0 \quad \mbox{in}\quad \mathcal{D}'(\Omega),
\end{equation}
the renormalized equation
\be\label{eq-r-v-ren}
 \dive\big(b(r)\vv\big)+\big(r b'(r) -b(r)\big)\dive \vv=0, \quad \mbox{holds in}\ \mathcal{D}'(  R^3),
\ee
for any $b\in C^0([0,\infty))\cap C^1((0,\infty))$ satisfying \eqref{b-pt1}-\eqref{b-pt2} provided $r$ and $\vv$ have been extended to be zero outside $\Omega$.
We also note that the hypothesis of the smoothness of $\Omega$ can be dropped provided (\ref{rrr}) is replaced by a stronger stipulation \eqref{wk-sl2}
for any $\psi \in C^{\infty} (\overline{\Omega})$.
\end{remark}

From the physical point of view, the available \emph{existence} theory of finite energy weak solutions in the sense of Definition \ref{def-weaksol} is still not completely satisfactory.
Recall that the relevant values of the adiabatic exponent are $1 \leq \gamma \leq 5/3$, where the case $\gamma = 1$ corresponds to the isothermal case while
$\gamma = 5/3$ is the adiabatic exponent of the monoatomic gas. Lions \cite{Lions-C} proves the existence of weak solutions in the range $\gamma > 5/3$. His proof is based on energy type arguments combined with the refined pressure estimates adopted also in the present paper. Lions' theory has been extended to the physical range
$\gamma \leq 5/3$ by several authors, see B\v rezina and Novotn\' y \cite{BREZNOV}, Plotnikov and Sokolowski \cite{PloSokbook} for the case $\g>3/2$,  Frehse, Steinhauer and Waigant \cite{FSW} for $\gamma > 4/3$. The best result available has been obtained by Plotnikov and Weigant in \cite{PW} for $\g>1$. All the results attacking the physical range $\gamma \leq 5/3$ use
delicate estimates that are not directly applicable to the case of perforated domains as they may fail to be uniform with respect to $\e \to 0$.

\subsection{Main results}

Our principal result concerns the construction of the inverse of the divergence operator on the family of perforated domains $\{ \O_\e \}_{\e > 0}$.

\subsubsection{Inverse of divergence}

\begin{theorem}\label{thm-bog} Let $\{ \O_\e \}_{\e > 0}$ be a family of domains enjoying the properties specified in {\rm Section \ref{sec:domain}}. Then there exists a linear operator
\[
\CalB_\e : L_0^{q}(\O_\e) \to W_0^{1,q}(\O_\e; R^3),\ 1 < q < \infty,
\]
such that for any $f\in L_0^{q}(\O_\e)$,
\ba\label{pt-bog}
&\dive \CalB_\e(f) =f \ \mbox{in} \ \O_\e,\\
&\|\CalB_\e(f)\|_{W_0^{1,q}(\O_\e; R^3)}\leq C\, \left(1+\e^{\frac{(3-q)\a-3}{q}}\right)\|f\|_{L^q(\O_\e)},
\ea
for some constant $C$ independent of $\e$.
\end{theorem}

The existence of such an operator on a \emph{fixed} Lipschitz domain has been established by several authors, notably by Bogovskii \cite{bog}.
Our contribution is therefore the explicit dependence of the estimate (\ref{pt-bog}) on $\e$. In particular, we recover a uniform bound as soon as
$(3-q)\a-3\geq0$. Note that the domains in the family $\{ \O_\e \}_{\e > 0}$ \emph{are not} uniformly Lipschitz domains or uniform John domains, for which such a result would follow from Bogovskii \cite{bog} and Galdi \cite{Galdi} or
Acosta, Dur\'an and Muschietti \cite{ADM} and  Diening, R{\r{u}}\v{z}i\v{c}ka and Schumacher \cite{DRS}. We also note thatTheorem \ref{thm-bog} is optimal with respect to the value of $q$ since functions in the Sobolev spaces $W^{1,q}$ with $ q>3$ are continuous and a uniform bound in (\ref{pt-bog}) is not expected if the holes become asymptotically dense and small in $\Omega$.

The proof of Theorem \ref{thm-bog} is given in Section \ref{sec:bog-proof}.

\subsubsection{Asymptotic limit of compressible fluid flows in perforated domains}

As a corollary of Theorem \ref{thm-bog}, we show that the asymptotic limit of solutions $[\vr_\e, \vu_\e]$ of the compressible Navier-Stokes system (\ref{i1})-(\ref{i4}), (\ref{mass}) in $\O_\e$ coincides with a solution of the same system on the homogeneous domain $\Omega$.

\begin{theorem}\label{thm-hom}
Suppose conditions \eqref{pp1}, \eqref{ex-force} and \eqref{mass} are satisfied. Suppose $2<\g\leq 3$ and  $\a>3$ be given such that
\be\label{res-g-a}
\a \,\frac{\g-2}{2\g-3} >1.
\ee
Let $[\vr_\e,\vu_\e]_{0<\e<1}$ be a family of finite energy weak solutions to \eqref{i1}-\eqref{i4} in $\O_\e$. Then we have uniform estimates
\be\label{est-r-u}
\sup_{0<\e<1}\left(\|\vr_\e\|_{L^{\b(\g)}(\O_\e)}+\|\vu_\e\|_{W_0^{1,2}(\O_\e; R^3)}\right)\leq C<\infty, \quad \beta(\gamma):=3(\g-1).
\ee
Moreover, extending $[\vr_\e, \vu_\e]$ to be zero outside $\Omega_\e$, we get, up to a substraction of subsequence,
\be\label{lim-r-u}
\vr_\e  \to \vr \ \mbox{weakly in} \ L^{\beta(\gamma)}(\O),\quad  \vu_\e \to \vu \   \mbox{weakly in}\  W_{0}^{1,2}(\O; R^3),
\ee
where $[\vr, \vu]$ is a finite energy weak solution to the same system of equations \eqref{i1}-\eqref{i4} in $\O$.

\end{theorem}

The proof of Theorem \ref{thm-hom}  is given in Section \ref{sec:pf-hom}.

\medskip

We give a remark concerning the similar result in two dimensional setting:
\begin{remark}

The argument in this paper cannot be directly extended to 2D setting. In particular, the construction of Bogovksii type operator in Section \ref{sec:bog-proof} and the choice of $g_\e$ in \eqref{def-ge1} in the proof of Lemma \ref{lem:momentum} do not apply to domains in $R^2$. However, similar result still may hold in 2D setting. Since in 2D setting, the density enjoys better integrability (in $L^{2\gamma}$ as long as $\g>1$) such that the pressure is in $L^2$. This allows us to apply the restriction operator constructed by Allaire \cite{ALL-NS1,ALL-NS2} to construct some uniformly bounded Bogovskii type operators, just as in the previous paper \cite{FL1}.  We employ again Allaire construction to find a function sequence that vanishes on the holes and converges to $1$ in some proper sense ($w^\e_k$ and $q^\e_k$ in Hypotheses H(1)-H(6) in Section 3.2).  So combining the techniques in this paper and in the previous paper \cite{FL1} should imply similar results in 2D setting. Of course, the holes must be much "smaller" - they have larger capacity in two dimensional spaces. Our interest here is to handle {\em better} $\gamma$ in 3D setting.

\end{remark}

We finally remark that, in this paper, the obstacles are assumed to be isolated in 3D domain. More realistic situation with connected boundaries may be treated in a similar manner for which the incompressible Stokes equations is considered in \cite{ALL-3}. However, such an extension is far from being trivial and a considerable
number of new difficulties would have to be overcome.

\section{Construction of the inverse of the divergence operator in perforated domains}
\label{sec:bog-proof}

This section is devoted to the proof of Theorem \ref{thm-bog}. For $f\in L^q(\Omega_\e)$ with $\int_{\Omega_\e} f\ \dx=0$,  we consider the extension $\tilde f=:E(f)$ defined as
\be\label{tilde-f}
\tilde f=f \ \mbox{in}\ \Omega_\e,\quad \tilde f=0 \ \mbox{on}\  \Omega\setminus \Omega_\e=  \bigcup_{n=1}^{N(\e)}  T_{\e,n}.
\ee
Clearly $\tilde f\in L_0^q(\O)$. Employing the standard Bogovskii's construction (see Galdi \cite[Chapter 3]{Galdi}) on the domain $\Omega$ we find $\vu={\mathcal B}_\O(\tilde f)\in W_0^{1,q}(\Omega; R^3)$ such that
\be\label{div-f-u}
\Div \vu=\tilde f \ \mbox{in} \ \Omega \quad \mbox{and} \quad \|\vu\|_{W_0^{1,q}(\Omega; R^3)}\leq C\,\|\tilde f\|_{L^q(\Omega)}= C\,\|f\|_{L^q(\Omega_\e)}
\ee
for some constant $C$ depending only on $\Omega$ and $q$.

In  accordance with hypotheses \eqref{dis-holes}, we introduce two cut-off functions $\chi_{\e,n}$ and $\phi_{\e,n}$ such that
\begin{equation}
\label{A1}
\chi_{\e,n} \in C^\infty_c (B(x_{\e,n}, \delta_2 \e )),\ \chi_{\e,n}|_{\overline{B(x_{\e,n}, \delta_1 \e)} } = 1,\
\| \nabla \chi_{\e,n} \|_{L^\infty( R^3; R^3)}\leq {C}{\e^{-1}},
\end{equation}
\begin{equation} \label{A2}
\phi_{\e,n} \in C^\infty_c (B(x_{\e,n}, \de_0 \e^\alpha )),\ \phi_{\e,n}|_{T_{\e,n}} = 1,\
\| \nabla \phi_{\e,n} \|_{L^\infty( R^3; R^3)}\leq {C}{\e^{-\alpha}},
\end{equation}

Denote
\[
D_{\e,n} = B(x_{\e,n}, \delta_2\e) \setminus \overline{ B(x_{\e,n}, \delta_1\e) },\   E_{\e,n} = B(x_{\e,n}, \delta_2\e)
\setminus T_{\e,n},\ F_{\e,n} = B(x_{\e,n}, \de_0 \e^\alpha) \setminus T_{\e,n}.
\]

Denoting
\[
\left< v \right>_{B} = \frac{1}{|B|} \int_B v \ {\rm d}x,
\]
we introduce
\ba\label{def-b-b}
&{\bf b}_{\e,n}(\vu) = \chi_{\e,n} \left( \vu - \left< \vu \right>_{D_{\e,n}} \right) \in W^{1,q}_0 (B(x_{\e,n}, \delta_2 \e); R^3), \\
&\boldsymbol\beta_{\e,n} (\vu) = \phi_{\e,n} \left< \vu \right>_{D_{\e,n}} \in W^{1,q}_0 (B(x_{\e,n}, \de_0 \e^\alpha );  R^3).
\ea

Revoking Poincar\' e's inequality
\[
\left\| \vu -  \left< \vu \right>_{D_{\e,n}} \right\|_{L^q (D_{\e,n};  R^3)} \leq C \e \| \nabla \vu \|_{L^q(D_{\e,n};  R^9)}
\]
and (\ref{A1}),
we estimate
\be \label{est-bk-bk1}
\left\| \nabla {\bf b}_{\e,n}(\vu) \right\|_{L^q (D_{\e,n};  R^9)} \leq
\left\| \chi_{\e,n} \nabla \left( \vu -  \left< \vu \right>_{D_{\e,n}} \right) \right\|_{L^q (D_{\e,n};  R^9)}  +
\left\| \nabla \chi_{\e,n} \otimes \left( \vu -  \left< \vu \right>_{D_{\e,n}} \right) \right\|_{L^q (D_{\e,n};R^9)}
\ee
\[
\leq C \left( \| \nabla \vu \|_{L^q(D_{\e,n}; R^3)} + \e^{-1} \left\| \vu -  \left< \vu \right>_{D_{\e,n}}  \right\|_{L^q (D_{\e,n})} \right)
\leq C \left\| \nabla \vu \right\|_{L^q(D_{\e,n}; R^3)}.
\]

Similarly, by virtue of (\ref{A2}) and Jensen's inequality,
\be \label{est-bk-bk2}
\| \nabla \beta_{\e,n} (\vu) \|_{L^q(B(x_{\e,n}, \de_0 \e^\alpha) ;  R^9))} = \left\| \nabla \phi_{\e,n} \cdot \left< \vu \right>_{D_{\e,n}} \right\|_{L^q(B(x_{\e,n}, \de_0 \e^\alpha))}
\ee
\[
\leq C \e^{ \left( \frac{3 } {q} - 1 \right) \alpha }  \left| \left< \vu \right>_{D_{\e,n}} \right| \leq C \e^{ \left( \frac{3 } {q} - 1 \right)\a -
\frac{3}{q} } \| \vu \|_{L^q(D_{\e,n};  R^3)}.
\]

Next, we claim the following result.

\begin{lemma}\label{lem-Bog1}
For any $1<q<\infty$, there exist a linear operator $\mathcal{B}_{E_{\e,n}}\,:\, L_0^q(E_{\e,n}) \to W_0^{1,q}(E_{\e,n}; R^3)$ and a linear operator $\mathcal{B}_{F_{\e,n}}\,:\, L_0^q(F_{\e,n}) \to W_0^{1,q}(F_{\e,n}; R^3)$ such that for any $f_1\in L_0^q(E_{\e,n})$ and any $f_2\in L_0^q(F_{\e,n})$, there holds
\ba\label{bog-E-F1}
&\Div \mathcal{B}_{E_{\e,n}}(f_1)=f_1,\quad \| \mathcal{B}_{E_{\e,n}}(f_1) \|_{W_0^{1,q}(E_{\e,n}; R^3)}\leq C\, \|f_1\|_{L^q(E_{\e,n})},\\
&\Div \mathcal{B}_{F_{\e,n}}(f_2)=f_2,\quad \| \mathcal{B}_{F_{\e,n}}(f_2) \|_{W_0^{1,q}(F_{\e,n}; R^3)}\leq C\, \|f_2\|_{L^q(F_{\e,n})},\nn
\ea
for some constant $C$ independent of $\e$ and   $n$  .

\end{lemma}

  There are several ways how to construct the operators $\mathcal{B}_{E_{\e,n}}$, $\mathcal{B}_{F_{\e,n}}$. We can use the construction of
Galdi \cite[Chapter 3]{Galdi} that mimics the original Bogovskii's proof. Note that this procedure yields indeed the operators with the corresponding
norm independent of $\e$ and $n$, see Galdi \cite[Chapter 3]{Galdi}. Alternatively, we observe that
both $E_{\e,n}$ and $F_{\e,n}$ are uniform families of John domains, whence the desired construction can be found in \cite{ADM} and \cite{DRS}. In the case
$1<q<3$, Lemma \ref{lem-Bog1} can be also shown by modifying the arguments of Allaire \cite[Lemma 2.2.4]{ALL-NS1}.


\medskip

We now define a restriction type operator in the following way:
\be\label{restr-1}
R_\e (\vu) := \vu- \sum_{n = 1}^{N(\e)}  \left({\bf b}_{\e,n}(\vu) -{\mathcal B}_{E_{\e,n}}(\Div {\bf b}_{\e,n} (\vu) ) \right)-\sum_{n=1}^{N(\e)}  \left(\boldsymbol\b_{\e,n}(\vu) -{\mathcal B}_{F_{\e,n}}(\Div \boldsymbol\b_{\e,n}(\vu)) \right),
\ee
where ${\mathcal B}_{E_{\e,n}}(\Div {\bf b}_{\e,n} (\vu) )$ and ${\mathcal B}_{F_{\e,n}}(\Div \b_{\e,n}(\vu))$ were extended to be zero outside $E_{\e,n}$ and $F_{\e,n}$, respectively. We say such an operator is of restriction type in the sprit of Tartar \cite{Tartar1} and Allaire \cite{ALL-NS1,ALL-NS2}, because, as we will see in the sequel argument, $R_\e$ is a well defined linear operator from $W^{1,p}_0(\O;R^3)$ to $W^{1,p}_0(\O_\e;R^3)$.

We first check that $R_\e (\vu)$ is well defined, specifically that
\be\label{zero-mean1}
\int_{E_{\e,n}} \Div {\bf b}_{\e,n}(\vu) \,\dx =0,\quad \int_{F_{\e,n}} \Div \boldsymbol \b_{\e,n} (\vu) \,\dx =0.
\ee
Indeed, on one hand, by \eqref{A1} and \eqref{def-b-b}, we have
\be\label{zero-mean2}
\int_{B(x_{\e,n}, \delta_2 \e)} \Div {\bf b}_{\e,n} (\vu) \,\dx =0,\quad \int_{B(x_{\e,n}, \de_0 \e^\alpha )} \Div \boldsymbol\b_{\e,n} (\vu) \,\dx =0.
\ee
On the other hand, by \eqref{div-f-u}, \eqref{A1} and \eqref{A2}, in particular, $\dive \vu =0$ on $T_{\e,n}$,  we have
\ba\label{zero-mean3}
&\Div {\bf b}_{\e,n} (\vu) = \chi_{\e,n} \Div \vu + \nabla  \chi_{\e,n} \cdot (\vu-\langle \vu \rangle_{D_{\e,n}})  =0, \ && \mbox{on} \quad T_{\e,n},\\
&\Div \boldsymbol \b_{\e,n}(\vu) = \nabla  \phi_{\e,n} \cdot \langle \vu\rangle_{D_{\e,n}}=0, \  &&\mbox{on} \quad T_{\e,n};
\ea
  whence (\ref{zero-mean1}) follows from  \eqref{zero-mean2} and (\ref{zero-mean3}).

\smallskip

By the definition of $R_\e (\vu)$ in \eqref{restr-1} and the property of   $\vu = {\mathcal B}_\O (\tilde f)$  claimed in \eqref{div-f-u}  , we have
\be\label{pt-Ru1}
R_\e (\vu) \in W_0^{1,q}(\O; R^3),\quad \Div R_\e (\vu)= \Div \vu =\tilde f \quad \mbox{in $\O$}.
\ee

Finally, we define the Bogovskii type operator ${\mathcal B}_\e$ through the composition of the extension operator, the classical Bogovskii operator, and the restriction operator defined above in \eqref{restr-1}:
\be\label{def-B-e}
{\mathcal B}_\e(f):= R_\e (\vu)= R_\e ({\mathcal B}_\O (\tilde f))= R_\e \circ {\mathcal B}_\O \circ E (f).
\ee
Our ultimate goal is to show that $\mathcal{B}_\e$ enjoys all the properties claimed in Theorem \ref{thm-bog}.

For any $x\in T_{\e,n}$, $1 \leq n \leq N(\e)$, we have
\ba\label{pt-Ru2}
R_\e (\vu)(x) &=\vu(x) - \Big({\bf b}_{\e,n} (\vu) -{\mathcal B}_{E_{\e,n}}(\Div {\bf b}_{\e,n} (\vu))\Big) (x) -  \Big(\boldsymbol\b_{\e,n}(\vu) -{\mathcal B}_{F_{\e,n}}(\Div \boldsymbol\b_{\e,n}(\vu)) \Big)(x) \\
 &= \vu(x) - {\bf b}_{\e,n}(\vu) (x) -\boldsymbol\b_{\e,n} (\vu)(x) \\
 &= \vu(x)- \chi_{\e,n}(x) \left(\vu(x)-\langle \vu\rangle_{D_{\e,n}}\right)- \phi_{\e,n} (x)\langle \vu\rangle_{D_{\e,n}}\\
 &=0,
\ea
where we used the fact that
$$
\chi_{\e,n}(x) = \phi_{\e,n}(x) =1, \quad \mbox{for any $x\in T_{\e,n}$}.
$$
Thus, we have shown the desired relations
\be\label{pt-Ru3}
R_\e (\vu) \in W_0^{1,q}(\O_\e; R^3),\quad \Div R_\e (\vu)= f \quad \mbox{in $\O_\e$}.
\ee

To finish the proof of Theorem \ref{thm-bog}, it remains to show the bound
\ba\label{pt-re-bog}
\|R_\e(\vu)\|_{W_0^{1,q}(\O_\e; R^3)}\leq C\, \left(1+\e^{\frac{(3-q)\a-3}{q}}\right)\|f\|_{L^q(\O_\e)}.
\ea

By \eqref{est-bk-bk1}, \eqref{est-bk-bk2} and Lemma \ref{lem-Bog1}, we have
\ba\label{est-fina-bog1}
&{\bf b}_{\e,n} (\vu) -{\mathcal B}_{E_{\e,n}}(\Div {\bf b}_{\e,n} (\vu) )\in W_0^{1,q}(B(x_{\e,n}, \delta_2 \e); R^3),\\
&\left\| {\bf b}_{\e,n}(\vu) -{\mathcal B}_{E_{\e,n}}(\Div {\bf b}_{\e,n} (\vu) ) \right\|_{W_0^{1,q}(B(x_{\e,n}, \delta_2 \e); R^3)}\leq C\,
\| \nabla \vu \|_{L^q(B(x_{\e,n}, \delta_2 \e); R^9)},
\ea
and
\ba\label{est-fina-bog2}
&\boldsymbol\b_{\e,n} (\vu) -{\mathcal B}_{F_{\e,n}}(\Div \boldsymbol\b_{\e,n} (\vu) )\in W_0^{1,q}(B(x_{\e,n}, \de_0 \e^\alpha) ; R^3),\\
&\left\| \boldsymbol\b_{\e,n} (\vu) -{\mathcal B}_{F_{\e,n}}(\Div \boldsymbol\b_{\e,n}(\vu) ) \right\|_{W_0^{1,q}(B(x_{\e,n}, \de_0 \e^\alpha); R^3)}\leq C\, \e^{\frac{(3-q)\a-3}{q}} \| \vu\|_{L^q(B(x_{\e,n}, \de_0 \e^\alpha); R^3)}.
\ea

Finally, by \eqref{div-f-u}, \eqref{restr-1} and the fact
$$
B(x_{\e,n_1}, \delta_2 \e ) \cap B(x_{\e,n_2}, \delta_2 \e ) =\emptyset,\quad \mbox{whenever $n_1\neq n_2$},
$$
a direct calculation implies the estimate \eqref{pt-re-bog}. We have completed the proof of Theorem \ref{thm-bog}.

\section{Asymptotic analysis of the compressible fluid flow on a family of perforated domains}\label{sec:pf-hom}
This section is devoted to the proof of Theorem \ref{thm-hom}.  For any $0<\e<1$, let $[\vr_\e,\vu_\e]$ be a finite energy weak solution
  satisfying the hypotheses of Theorem \ref{thm-hom}. By the known results concerning integrability of weak solutions to the stationary
Navier-Stokes system, we have, see e.g. Novotn\' e and Stra\v skraba \cite[Chapter 4]{N-book}:
\be\label{re-ve-space}
\vr_\e \in L^{\beta(\g)}(\O_\e),\ \b(\g)=3(\g-1); \quad \vu_\e \in W_0^{1,2}(\O_\e; R^3).
\ee
  As we assume $2<\g\leq 3$, we have $\b(\g)=3(\g-1) >3$; whence, by Remark \ref{rem:hom2}, the solution $[\vr_\e,\vu_\e]$ is also a renormalized weak solution:
\begin{lemma}\label{lem:renormal}
We have
\be\label{eq-renomal}
\dive\big(b(\tilde\vr_\e)\tilde\vu_\e\big)+\big(\tilde\vr_\e b'(\tilde\vr_\e) -b(\tilde\vr_\e)\big)\dive \tilde\vu_\e=0, \quad \mbox{in}\ \mathcal{D}'(  R^3),
\ee
for any $b\in C^0([0,\infty))\cap C^1((0,\infty))$ satisfying \eqref{b-pt1}-\eqref{b-pt2}, where $[\tilde \vr_\e, \tilde \vu_\e]$ denotes the functions $[\vr, \vu]$ extended
to be zero outside $\O_\e$.
\end{lemma}

\subsection{Uniform estimates}

We have the solution $[\vr_\e,\vu_\e]$ is in the function spaces shown in \eqref{re-ve-space}, but the classical estimates of their norms depend on the domain $\O_\e$, in particular on the Lipschitz character of $\O_\e$ which is unbounded as $\e\to 0$. To show the uniform estimates \eqref{est-r-u}, we need to employ the uniform Bogovskii type operator $\CalB_\e$ obtained in Theorem \ref{thm-bog} and constructed in Section \ref{sec:bog-proof}.

By using  the  Korn's inequality and H\"older's inequality, the energy inequality \eqref{wk-sl4} implies
\ba\label{ub1}
 \|\nabla  \vu_\e \|^2_{L^2(\O_\e; R^9)} &\leq C\, \Big(\|\vf\|_{L^\infty(\O; R^3)} \|\vr_\e \|_{L^{\frac{6}{5}}(\O_\e)} \|\vu_\e \|_{L^6(\O_\e; R^3)}\\
 &\qquad+\|\vg\|_{L^\infty(\O; R^3)}\|\vu_\e \|_{L^1(\O_\e; R^3)}\Big).
\ea

Since $\vu_\e\in W_0^{1,2}(\O_\e; R^3)$ has zero trace on the boundary, the Sobolev embedding inequality implies
\be\label{ub2}
\|\vu_\e\|_{L^6(\O_\e; R^3)}\leq C\, \ \|\nabla \vu_\e\|_{L^2(\O_\e; R^9)},
\ee
for some constant $C$ independent of the domain $\O_\e$.

By the above two estimates in \eqref{ub1} and \eqref{ub2}, we deduce
\ba\label{ub3}
\|\nabla  \vu_\e \|_{L^2(\O_\e; R^9)}+\|\vu_\e\|_{L^6(\O_\e; R^3)} &\leq C\, \left(\|\vf\|_{L^\infty(\O_\e; R^3)} \|\vr_\e \|_{L^{\frac{6}{5}}(\O_\e)} +\|\vg\|_{L^\infty(\O_\e; R^3)}\right) \\
&\leq C\, \left( \|\vr_\e \|_{L^{\frac{6}{5}}(\O_\e)} +1\right).
\ea

\smallskip

Let $\mathcal{B}_\e$ is the  operator introduced in Theorem \ref{thm-bog}, we define the test function
\be\label{test-Be}
\varphi:=\mathcal{B}_\e\left(\vr_\e^{2\gamma-3}-\frac{1}{|\O_\e|}\intOe {\vr_\e^{2\gamma-3}}\right).
\ee
By \eqref{re-ve-space} and $2<\gamma\leq 3$, we have
\be\label{re-space}
\vr_\e^{2\gamma-3} \in L^{\frac{3\g-3}{2\g-3}}(\O_\e),\quad 2 \leq \frac{3\g-3}{2\g-3}<3.
\ee
Then by Theorem \ref{thm-bog}, we have
\ba\label{est-phi}
&\dive \varphi = \vr_\e^{2\gamma-3}-\frac{1}{|\Omega_\e|}\intOe {\vr_\e^{2\gamma-3}}\quad \mbox{in}\ \O_\e,\\
&\|\varphi\|_{W^{1,\frac{3\g-3}{2\g-3}}_0(\O_\e; R^3)}\leq C\,(1+\e^{\sigma_1}) \left(\|\vr_\e^{2\g-3}\|_{L^{\frac{3\g-3}{2\g-3}}(\O_\e)}+ \|\vr_\e^{2\g-3} \|_{L^1(\O_\e)}\right) \\
 &\qquad \qquad \qquad \qquad \leq C\, \|\vr_\e\|^{{2\g-3}}_{L^{3\gamma-3}(\O_\e)},
\ea
where
$$
\sigma_1:=\frac{\left(3-\frac{3\g-3}{2\g-3}\right)\a-3}{\frac{3\g-3}{2\g-3}}=\frac{2\g-3}{\g-1} \left(\frac{\g-2}{2\g-3}\cdot \a-1\right)>0,
$$
for which the positivity is guaranteed by condition \eqref{res-g-a}.

Taking $\vp$ as a test function in the weak formulation of the momentum equation \eqref{wk-sl3} gives
\be\label{prr}
\intOe{p(\vr_\e)\vr_\e^{2\g-3}} = \sum_{ j=1}^4 I_j
\ee
with
\ba\label{Ij}
&I_1:=\intOe{p(\vr_\e)} \ \frac{1}{|\O_\e|}\intOe {\vr_\e^{2\gamma-3}}, \\
&I_2:=\intOe{\mu \nabla \vu_\e:\nabla \varphi}+\intOe{\left(\frac{\mu}{3} + \eta \right) \dive \vu_\e \dive \varphi},\\
&I_3:=-\intOe{\vr_\e \vu_\e\otimes\vu_\e:\nabla \varphi},\\
& I_4:=-\intOe{(\vr_\e \vf +\vg)\cdot\varphi}.\nn
\ea

For $I_1$:
\ba\label{I1}
&I_1:=\intOe{p(\vr_\e)} \ \frac{1}{|\O_\e|}\intOe {\vr_\e^{2\gamma-3}}\leq C\,\|\vr_\e\|_{L^\g(\O_\e)}^\g \|\vr_\e\|_{L^{2\g-3}(\O_\e)}^{2\g-3} \\
&\quad \leq C\,\left(\|\vr_\e\|_{L^{1}(\O_\e)}^{(1-\th_1)\g}\|\vr_\e\|_{L^{3\g-3}(\O_\e)}^{\th_1\g} \right) \left(\|\vr_\e\|_{L^{1}(\O_\e)}^{(1-\th_2)(2\g-3)}\|\vr_\e\|_{L^{3\g-3}(\O_\e)}^{\th_2(2\g-3)}\right)\\
&\quad \leq C\,M^{(1-\th_1)\g+(1-\th_2)(2\g-3)} \|\vr_\e\|_{L^{3\g-3}(\O_\e)}^{\th_1\g+\th_2(2\g-3)}\\
&\quad \leq C\,\left( 1 + \|\vr_\e\|_{L^{3\g-3}(\O_\e)}^{\max\{\th_1,\th_2\}(3\g-3)}\right),\nn
\ea
where we used \eqref{mass}, Young's inequality, and interpolations between Lebesgue spaces:
$$
0<\th_1,\th_2<1\ \  \mbox{s.t.}\ \  \frac{1}{\g}=(1-\th_1)+\frac{\th_1}{3\g-3},\quad  \frac{1}{2\g-3}=(1-\th_2)+\frac{\th_2}{3\g-3}.
$$

For $I_2$:
\ba\label{I2}
I_2 &\leq  C\, \|\nabla \vu_\e\|_{L^2(\O_\e; R^9)}\|\nabla\varphi\|_{L^2(\O_\e; R^9)}\leq C\, \left( \|\vr_\e \|_{L^{\frac{6}{5}}(\O_\e)} +1\right)\|\nabla\varphi\|_{L^{\frac{3\g-3}{2\g-3}}(\O_\e; R^9)} \\
&\leq C\, \left( \|\vr_\e \|_{L^{\frac{6}{5}}(\O_\e)} +1\right)\|\vr_\e\|_{L^{3\g-3}(\O_\e)}^{{2\g-3}}\\
&\leq C\, \left( M^{(1-\th_3)} \|\vr_\e \|_{L^{3\g-3}(\O_\e)}^{\th_3}+1  \right)  \|\vr_\e \|_{L^{3\g-3}(\O_\e)}^{2\g-3}\\
&\leq C \left( 1 +  \|\vr_\e \|_{L^{3\g-3}(\O_\e)}^{2\g-2} \right),\nn
\ea
where we used \eqref{ub3}, \eqref{re-space} and \eqref{est-phi}. The number $0<\th_3<1$ is determined by
$$
 \frac{5}{6}=(1-\th_3)+\frac{\th_3}{3\g-3}.
$$

For $I_3$:
\ba\label{I3}
I_3 &\leq  C\, \|\vr_\e\|_{L^{3\g-3}(\O_\e)} \| \vu_\e\|_{L^6(\O_\e; R^3)}^2 \|\nabla\varphi\|_{L^{\frac{3\g-3}{2\g-3}}(\O_\e; R^9)}\\
&\leq C\, \|\vr_\e\|_{L^{3\g-3}(\O_\e)} \left(1+\| \vr_\e\|_{L^{6/5}(\O_\e)}^2 \right) \|\vr_\e\|_{L^{3\g-3}(\O_\e)}^{2\g-3}\\
&\leq C\, \|\vr_\e\|_{L^{3\g-3}(\O_\e)}^{2\g-2} \left( 1 + \| \vr_\e\|_{L^{3\g-3}(\O_\e)}^{2\th_4} \right)\\
&\leq C\, \left( 1 + \|\vr_\e\|_{L^{3\g-3}(\O_\e)}^{(3\g-3)-(\g-1-2\th_4)}  \right),\nn
\ea
where
$$
0<\th_4<1 \quad \mbox{s.t.}\quad \frac{5}{6}=(1-\th_4)+\frac{\th_4}{3\g-3}.
$$
This implies
$$
\th_4=\frac{\g-1}{2(3\g-4)},\quad (\g-1-2\th_4)= (\g-1)\ \frac{3\g-5}{3\g-4}>0.
$$

For $I_4$:
\be\label{I4}
I_4\leq  C\, \left(\|\vr_\e\|_{L^{2}(\O_\e)}+1\right) \| \vp\|_{L^2(\O_\e; R^3)} \leq  C\, \left(1+\|\vr_\e\|_{L^{3\g-3}(\O_\e)}^{2\g-2}\right).\nn
\ee

Summing up the estimates in \eqref{prr}-\eqref{I4} implies
\be\label{est-prr}
\|\vr_\e\|_{L^{3\g-3}(\O_\e)}^{3\g-3}   \leq  C\, \left(1+\|\vr_\e\|_{L^{3\g-3}(\O_\e)}^{\b_1(\g)}\right),\nn
\ee
for some $\b_1(\g)<3\g-3$. Then we deduce
\be\label{est-r-u0}
\|\vr_\e\|_{L^{3\g-3}(\O_\e)}  \leq  C;\quad \mbox{moreover, by \eqref{ub3},}\  \|\vu_\e\|_{W_0^{1,2}(\O_\e; R^3)} \leq C.
\ee

Let $[\tilde \vr_\e,\tilde \vu_\e]$ be the zero extension of $[\vr_\e, \vu_\e]$ in $\O$. Then by \eqref{est-r-u0} we have
\be\label{est-r-u-tilde}
\|\tilde \vr_\e\|_{L^{3\g-3}(\O)}+   \|\tilde \vu_\e\|_{W_0^{1,2}(\O; R^3)}\leq  C.
\ee
Thus, up to a substraction of subsequence,
\be\label{r-u-limit0}
\tilde \vr_\e  \to \vr \ \mbox{weakly in} \ L^{3\g-3}(\O),\quad \tilde \vu_\e \to \vu \   \mbox{weakly in}\  W_{0}^{1,2}(\O; R^3).
\ee

We obtained the uniform estimate \eqref{est-r-u} and the weak convergence in \eqref{lim-r-u}.

\subsection{Equations in homogeneous domain}

In this section, we deduce the equations in $[\tilde \vr_\e,\tilde \vu_\e]$ and $[\vr,\vu]$ in the homogeneous domain $\O$.

First, the fact that $[\tilde \vr_\e,\tilde \vu_\e]$ is a renormalized weak solution (see Lemma \ref{lem:renormal}) implies that  $[\tilde \vr_\e,\tilde \vu_\e]$ solves \eqref{eq-renomal}.

Next we claim that the couple $[\tilde\vr_\e,\tilde \vu_\e]$ solves the same momentum equations as \eqref{i2} in $\O$ up to a small remainder:
\begin{lemma}\label{lem:momentum}
Under the assumptions in {\rm Theorem \ref{thm-hom}},  there holds
 \be\label{eq-monmentum}
\dive(\tilde\vr_\e\tilde \vu_\e\otimes\tilde \vu_\e)+\nabla p(\tilde\vr_\e)=\dive\SSS(\nabla \tilde\vu_\e)+\tilde\vr_\e \vf + \vg+{\rm r}_\e,\quad \mbox{in}\ \mathcal{D}'(\O; R^3),
 \ee
where the distribution ${\rm r}_\e$ is small in the following sense:
\be\label{est-re}
|\langle r_\e,\varphi\rangle_{\mathcal{D}'(\O; R^3),\mathcal{D}(\O; R^3)}|\leq C\, \e^{\de_1} \left(\|\nabla\varphi\|_{L^{\frac{3\g-3}{2\g-3}+\delta_0}(\O_\e; R^9)} + \|\vp\|_{L^{r_1}(\O; R^3)}\right),
\ee
where $\de_0>0$ is chosen such that \eqref{est-Ie1-2} is satisfied, $1<r_1<\infty$ is determined by \eqref{def-de-r1} and $\de_1>0$ is defined in \eqref{def-de1} later on.
\end{lemma}

\begin{proof}[Proof of {\rm Lemma \ref{lem:momentum}}]

By the assumption on the distribution and size of the holes in \eqref{dis-holes}, there exists $g_\e\in C^\infty(\O)$ satisfying $0\leq g_\e \leq 1$ and
\be\label{def-ge1}
g_\e =0 \ \mbox{on}\ \bigcup_{n=1}^{N(\e)}  T_{\e,n},\quad g_\e=1 \  \mbox{in} \ \O\setminus \bigcup_{n=1}^{N(\e)} \overline{B(x_{\e,n},\de_0\e^\a)},\quad \|\nabla g_\e\|_{L^\infty(\O; R^3)}\leq C\, \e^{-\alpha}.
\ee
Direct calculation gives that for any $1\leq q\leq \infty$:
\be\label{def-ge2}
\|1-g_\e\|_{L^q(\O)}\leq C\, \e^{\frac{3(\alpha-1)}{q}},\quad \|\nabla g_\e\|_{L^q(\Omega; R^3)}\leq C\ \e^{\frac{3(\alpha-1)}{q}-\alpha}.
\ee

Then for any $\varphi\in C_c^\infty (\O; R^3)$, we have
\ba\label{wk-eq-tvu}
&\intO{\tilde\vr_\e\tilde \vu_\e\otimes\tilde \vu_\e:\nabla \varphi+ p(\tilde\vr_\e) \,\dive \varphi-\SSS(\nabla \tilde\vu_\e):\nabla \varphi + \tilde\vr_\e  \vf \cdot \varphi + \vg \cdot \varphi}\\
&\quad =\int_{\O_\e}\Big(\tilde\vr_\e\tilde \vu_\e\otimes\tilde \vu_\e:\nabla (\varphi g_\e)+ p(\tilde\vr_\e)\, \dive (\varphi g_\e)-\SSS(\nabla \tilde\vu_\e):\nabla (\varphi g_\e) \\
&\qquad\qquad + \tilde\vr_\e \vf \cdot (\varphi g_\e) + \vg \cdot (\varphi g_\e)\Big)\,\dx + I_\e\\
&\quad =I_\e,\nn
\ea
where we used the fact $\vp g_\e \in C_c^\infty (\O_\e; R^3)$ is a good test function for the momentum equations \eqref{i2} in $\O_\e$, and the quantity $I_\e$ is of the form
\ba\label{def-Ie}
&I_\e:=\sum_{j=1}^4 I_{j,\e},\quad \mbox{with:}\\
&I_{1,\e}:=  \intO{ \tilde \vr_\e \tilde \vu_\e\otimes \tilde \vu_\e:(1-g_\e)\nabla \varphi -\tilde \vr_\e \tilde \vu_\e\otimes \tilde \vu_\e:(\nabla g_\e\otimes  \varphi) },\\
&I_{2,\e}:= \intO{  p(\tilde \vr_\e)(1-g_\e)\dive \varphi -p(\tilde \vr_\e) \nabla g_\e\cdot \varphi},\\
&I_{3,\e}:=  \intO{ -\SSS(\nabla \tilde \vu_\e):(1-g_\e)\nabla \varphi +\SSS(\nabla \tilde \vu_\e):(\nabla g_\e\otimes  \varphi)},\\
&I_{4,\e}:= \intO{  \tilde \vr_\e  \vf \cdot (1-g_\e)\varphi +\vg \cdot (1-g_\e)  \varphi}.
\ea
We now estimate $I_{j,\e}$ one by one. For $I_{1,\e}$, direct calculation gives
\ba\label{est-Ie1}
|I_{1,\e}| &\leq  C\, \|\vr_\e\|_{L^{3\g-3}(\O_\e)} \| \vu_\e\|_{L^6(\O_\e; R^3)}^2 \left( \|(1-g_\e)\nabla\varphi\|_{L^{\frac{3\g-3}{2\g-3}}(\O_\e; R^9)}+ \|\nabla g_\e\otimes \varphi\|_{L^{\frac{3\g-3}{2\g-3}}(\O_\e; R^9)}\right)\\
&\quad \leq C\, \left(\|1-g_\e\|_{L^{r_1}(\O)}\|\nabla\varphi\|_{L^{\frac{3\g-3}{2\g-3}+\delta_0}(\O_\e; R^9)} + \|\nabla g_\e\|_{L^{\frac{3\g-3}{2\g-3}+\delta_0}(\O_\e; R^3)} \|\vp\|_{L^{r_1}(\O; R^3)} \right),\nn
\ea
where
\be\label{def-de-r1}
0<\de_0<1, \quad 1<r_1<\infty,\quad \frac{1}{r_1}+ \left(\frac{3\g-3}{2\g-3}+\delta_0\right)^{-1}=\frac{2\g-3}{3\g-3}.
\ee
By \eqref{def-ge2}, we have
\be\label{est-Ie1-0}
\|1-g_\e\|_{L^{r_1}(\O)}\leq C\, \e^{\frac{3(\alpha-1)}{r_1}}, \quad \|\nabla g_\e\|_{L^{\frac{3\g-3}{2\g-3}+\de_0}(\O; R^3)}\leq C\, \e^{3(\alpha-1)\left(\frac{3\g-3}{2\g-3}+\delta_0\right)^{-1}-\alpha}.
\ee
We calculate
\be\label{est-Ie1-1}
3(\alpha-1)\left(\frac{3\g-3}{2\g-3}\right)^{-1}-\alpha = \frac{\a\g-2\a-2\g+3}{\g-1}>0,
\ee
where we used the condition \eqref{res-g-a} which is equivalent to
$$
\a\g-2\a-2\g+3>0.
$$

Then by \eqref{est-Ie1-0} and \eqref{est-Ie1-1}, we can choose $\de_0>0$ small enough such that
\be\label{est-Ie1-2}
3(\alpha-1)\left(\frac{3\g-3}{2\g-3}+\delta_0\right)^{-1}-\alpha =:h(\de_0)>0.
\ee

We finally obtain
\be\label{est-Ie1-f}
|I_{1,\e}| \leq C\, \e^{\de_1} \left(\|\nabla\varphi\|_{L^{\frac{3\g-3}{2\g-3}+\delta_0}(\O_\e; R^9)} + \|\vp\|_{L^{r_1}(\O; R^3)}\right),
\ee
where
\be\label{def-de1}
 \de_1:=\min\left\{\frac{3(\alpha-1)}{r_1},h(\de_0)\right\}>0,
\ee
 where $\de_0>0$ is chosen such that \eqref{est-Ie1-2} is satisfied and $1<r_1<\infty$ is determined by \eqref{def-de-r1}.

\medskip

For $I_{2,\e}$, similar as the estimate for $I_{1,\e}$, we have
\ba\label{est-Ie2}
|I_{2,\e}| &\leq  C\, \|p(\tilde \vr_\e)\|_{L^{\frac{3\g-3}{\g}}(\O_\e)}\left( \|(1-g_\e)\dive \varphi\|_{L^{\frac{3\g-3}{2\g-3}}(\O_\e)}+ \|\nabla g_\e\cdot  \varphi\|_{L^{\frac{3\g-3}{2\g-3}}(\O_\e)}\right)\\
&\quad \leq C\, \e^{\de_1} \left(\|\nabla\varphi\|_{L^{\frac{3\g-3}{2\g-3}+\delta_0}(\O_\e; R^9)} + \|\vp\|_{L^{r_1}(\O; R^3)}\right).
\ea

\medskip

For $I_{3,\e}$ and $I_{4,\e}$, the similar argument gives the following non-optimal estimate:
\ba\label{est-Ie3}
|I_{3,\e}|+|I_{4,\e}|  \leq C\, \e^{\de_1} \left(\|\nabla\varphi\|_{L^{\frac{3\g-3}{2\g-3}+\delta_0}(\O_\e; R^9)} + \|\vp\|_{L^{r_1}(\O; R^3)}\right).
\ea

\medskip

Summing up the estimates in \eqref{est-Ie1-f},  \eqref{est-Ie2} and \eqref{est-Ie3} implies \eqref{est-re}.

\end{proof}

\subsection{The limit equations}

This section is devoted to deduce the equations in the limit couple $[\vr,\vu]$ obtained in \eqref{r-u-limit0}. First of all, by compact Sobolev embedding, we have
\be\label{st-con-vue}
\tilde\vu_\e \to \vu \quad \mbox{strongly in}\quad L^q(\Omega; R^3) \quad \mbox{for any $1\leq q<6$}.
\ee
Thus, there holds the weak convergence of nonlinear terms:
\ba\label{wl-ru1}
& \tilde\vr_\e \tilde\vu_\e \to \vr\vu \quad &&\mbox{weakly in}\quad L^q(\O; R^3) \quad \mbox{for any $1< q<\frac{6\gamma-6}{\gamma+1}$},\\
& \tilde\vr_\e \tilde\vu_\e\otimes\tilde\vu \to \vr\vu\otimes\vu \quad &&\mbox{weakly in}\quad L^q(\O; R^9) \quad \mbox{for any $1< q<\frac{3\gamma-3}{\gamma}$}.
\ea

Then passing $\e\to 0$ in \eqref{eq-renomal} and in \eqref{eq-monmentum} gives
\ba\label{eq-vu-vr-0}
&\dive (\vr \vu) = 0,\\
&\dive (\vr \vu \otimes \vu) +\nabla  \overline {p(\vr)} =\dive \SSS(\nabla \vu)+\vr \vf+ \vg,
\ea
in the sense of distribution in ${\mathcal D}'(\O)$, where $\overline {p(\vr)}$ is the weak limit of $p(\tilde \vr_\e)$ in $L^{\frac{3\g-3}{\g}}(\O)$. Moreover, by Remark \ref{rem:hom2}, $[\vr,\vu]$ satisfies the renomalized equation
\be\label{eq-renomal-lim}
\dive \big( b(\vr)\vu \big) + \big( \vr b'(\vr) - b(\vr) \big) \dive \vu = 0, \quad \mbox{in}\ \mathcal{D}'( R^3),
\ee
where $b\in C^0([0,\infty))\cap C^1((0,\infty))$ satisfies \eqref{b-pt1}-\eqref{b-pt2}.

To finish the proof of Theorem \ref{thm-hom}, it is left to show $\overline {p(\vr)} = p(\vr)$. This is done in the next section.

\subsection{Convergence of pressure term - end of the proof}

We introduce the so-called effective viscous flux $p(\vr)-(\frac{4\mu}{3}+\eta)\dive \vu$ enjoying some compactness property given in the following lemma. This property plays a crucial role in the existence theory of weak solutions for the compressible Navier-Stokes equations.

\begin{lemma}\label{lem:flux}
Up to a substraction of subsequence, there holds for any $\psi\in C_c^\infty(\Omega)$:
\be\label{flux}
\lim_{\e\to 0}\intO{\psi \left(p(\tilde\vr_\e)-(\frac{4\mu}{3}+\eta)\dive \tilde\vu_\e\right)\tilde\vr_\e}=\intO{\psi\left(\overline{p(\vr)}-(\frac{4\mu}{3}+\eta)\dive \vu\right)\vr}.
\ee
\end{lemma}
\begin{proof}[Proof of {\rm Lemma \ref{lem:flux}}]
The proof of Lemma \ref{lem:flux} is quite tedious but nowadays well understood. The main idea is to employ proper test functions by employing Fourier multiplier and Riesz type operators. We refer to Section 1.3.7.2 in \cite{N-book} or Section 10.16 in \cite{F-N-book} for the definitions and properties used here of Fourier multiplier and Riesz operators.  These proper test functions are defined by
\be\label{def-test-flux}
\psi \nabla \Delta^{-1}(1_{\O}\tilde\vr_\e), \quad  \psi \nabla \Delta^{-1}(1_{\O}\vr),
\ee
where $\psi\in C_c^\infty(\O)$ and $\Delta^{-1}$ is the Fourier multiplier on $ R^3$ with symbol $-\frac{1}{|\xi|^2}$.

We observe that
$$
\nabla \nabla \Delta^{-1}=\left(\mathcal{R}_{i,j}\right)_{1\leq i,j\leq 3}
$$
are the classical Riesz operators. Then for any $f\in L^q( R^3),~1<q<\infty$:
\be\label{est-test-flux1}
\|\nabla\nabla \Delta^{-1}(f)\|_{L^q( R^3; R^9)}\leq C\, \|f\|_{L^q( R^3)}.\nn
\ee
By the embedding theorem in homogeneous Sobolev spaces (see Theorem 1.55 and Theorem 1.57 in \cite{N-book} or Theorem 10.25 and Theorem 10.26 in \cite{F-N-book}), we have for any $f\in L^q( R^3), \ \supp\, f\subset \O$:
\ba\label{est-test-flux2}
&\|\nabla \Delta^{-1}(f)\|_{L^{q^*}( R^3; R^3)}\leq C\, \| f \|_{L^q( R^3)}\quad \frac{1}{q^*}=\frac{1}{q}-\frac{1}{3}, \ \mbox{if $1<q<3$},\\
&\|\nabla \Delta^{-1}(f)\|_{L^{q^*}( R^3; R^3)}\leq C\, \| f \|_{L^q( R^3)}\quad \mbox{for any} \ q^*<\infty,  \mbox{if $q\geq3$}.\nn
\ea
Then by the uniform estimate for $\tilde \vr_\e$ and its weak limit $\vr$ in \eqref{r-u-limit0} and the fact $3\g-3>3$ under our assumption $\g>2$, we have for any $q<\infty$:
\ba\label{est-test-flux3}
&\|\psi \nabla \Delta^{-1}(1_{\O}\tilde\vr_\e)\|_{L^q(\O; R^3)} + \|\psi \nabla \Delta^{-1}(1_{\O}\vr)\|_{L^q(\O; R^3)}\leq C,\\
&\|\nabla \left(\psi \nabla \Delta^{-1}(1_{\O}\tilde\vr_\e)\right)\|_{L^{3\g-3}(\O; R^9)} + \|\nabla \left(\psi \nabla \Delta^{-1}(1_{\O}\vr)\right)\|_{L^{3\g-3}(\O; R^9)}\leq C.
\ea

Since $2<\g\leq 3$, we have $ 3\g-3 >\frac{3\g-3}{2\g-3} $. Then choosing $\de_0>0$ in Lemma \ref{lem:momentum} to be small enough, we have
$$
3\g-3 \geq \frac{3\g-3}{2\g-3}+\de_0.
$$
Thus, \eqref{est-re} and \eqref{est-test-flux3} implies
\ba\label{est-test-flux4}
&|\langle r_\e,\psi \nabla \Delta^{-1}(1_{\O}\tilde\vr_\e)\rangle_{\mathcal{D}'(\O; R^3),\mathcal{D}(\O; R^3)}|\\
&\leq C\, \e^{\de_1} \left(\|\nabla \left(\psi \nabla \Delta^{-1}(1_{\O}\tilde\vr_\e)\right)\|_{L^{3\g-3}(\O_\e; R^9)} + \|\psi \nabla \Delta^{-1}(1_{\O}\tilde\vr_\e)\|_{L^{r_1}(\O; R^3)}\right)\\
&\leq C\,\e^{\de_1},\nn
\ea
which goes to zero as $\e\to 0$.

\smallskip

Now we chose $\psi \nabla \Delta^{-1}(1_{\Omega}\tilde\vr_\e)$ as a test functions in the weak formulation of equation \eqref{eq-monmentum} and pass $\e\to 0$. Then we choose $\psi \nabla \Delta^{-1}(1_{\Omega}\vr)$ as a test functions in the weak formulation of $\eqref{eq-vu-vr-0}_2$. By comparing the results of theses two operations, through long but straightforward calculations, we obtain that
\ba\label{flux-l1}
I:&=\lim_{\e\to 0}\intO{\psi \left(p(\tilde\vr_\e)-(\frac{4\mu}{3}+\eta)\Div \tilde\vu_\e\right)\tilde\vr_\e}-\intO{\psi\left(\overline{p(\vr)}-(\frac{4\mu}{3}+\eta)\Div \vu\right)\vr}\\
&=\lim_{\e\to 0}\intO{\tilde\vr_\e\tilde\vu_\e^i\tilde\vu_\e^j \psi \mathcal{R}_{i,j}(1_{\Omega} \tilde\vr_\e)}-\intO{\vr \vu^i\vu^j \psi \mathcal{R}_{i,j}(1_{\Omega} \vr)}.
\ea

On the other hand, choosing $1_{\Omega} \nabla \Delta^{-1}(\psi\tilde\vr_\e \tilde\vu_\e)$ as a test function in the weak formulation of \eqref{eq-renomal} with $b(\vr)=\vr$ and  $1_{\Omega} \nabla \Delta^{-1}(\psi\vr \vu)$ as a test function in the weak formulation of $\eqref{eq-vu-vr-0}_1$ implies
\be\label{flux-l2}
\intO{1_{\Omega}\tilde\vr_\e \tilde\vu_\e^i  \mathcal{R}_{i,j}(\psi\tilde\vr_\e \tilde\vu_\e) }=0,\quad \intO{1_{\Omega}\vr \vu^i  \mathcal{R}_{i,j}(\psi \vr \vu)}=0.
\ee

Plugging (\ref{flux-l2}) into (\ref{flux-l1}) yields
\ba\label{flux-l3}
&I=\lim_{\e\to 0}\intO{\tilde\vu_\e^i\Big(\tilde\vr_\e\tilde\vu_\e^j \psi \mathcal{R}_{i,j}(1_{\Omega} \tilde\vr_\e)-1_{\Omega}\tilde\vr_\e   \mathcal{R}_{i,j}(\psi\tilde\vr_\e \tilde\vu_\e)\Big)}\\
&\qquad\qquad-\intO{ \vu^i\Big(\vr \vu^j \psi \mathcal{R}_{i,j}(1_{\Omega} \vr)-1_{\Omega}\vr  \mathcal{R}_{i,j}(\psi \vr \vu)\Big)}.\nn
\ea

We introduce the following lemma, which is a variance of the divergence-curl lemma, and we refer to \cite[Lemma 3.4]{FNP} for the proof.
\begin{lemma}\label{lem-ct}Let $1<p,q<\infty$ satisfy $$\frac{1}{r}:=\frac{1}{p}+\frac{1}{q}<1.$$
Suppose
\[
u_\e \to u \quad\mbox{weakly in}\quad L^p( R^3),\quad v_\e \to v \quad\mbox{weakly in}\quad L^q( R^3),\ \mbox{as $\e \to 0$}.
\]
Then for any  $1\leq i,j\leq 3$:
\[
u_\e \mathcal{R}_{i,j}(v_\e)-v_\e \mathcal{R}_{i,j}(u_\e) \to u \mathcal{R}_{i,j}(v)-v \mathcal{R}_{i,j}(u) \quad\mbox{weakly in}\quad L^r( R^3).
\]
\end{lemma}

The convergence result \eqref{flux} can be deduced by the strong convergence of the velocity in \eqref{st-con-vue} and Lemma \ref{lem-ct}.

\end{proof}

\medskip

A direct consequence of the compactness of the effective viscous flux is the following:
\begin{lemma}\label{lem:flux2} We denote $\overline {p(\vr)\vr}$ as the weak limit of $p(\tilde \vr_\e)\tilde \vr_\e$ in $L^{\frac{3\g-3}{\g+1}}(\O)$. Then $\overline {p(\vr)\vr}=\overline {p(\vr)}\vr$ .

\end{lemma}

\begin{proof}[Proof of {\rm Lemma \ref{lem:flux2}}]
First of all, we have
$$ (3\g-3)-(\g+1)=2\g -4 >0 . $$
Then by \eqref{est-r-u-tilde}, we have
$$
p(\tilde \vr_\e)\tilde \vr_\e \to \overline {p(\vr)\vr} \quad \mbox{weakly in}\quad L^{\frac{3\g-3}{\g+1}}(\O).
$$

Taking $b(s)=s\log s$ in the renormalized equations \eqref{eq-renomal} and \eqref{eq-renomal-lim} implies
\be\label{eq:slogs}
\dive\big((\tilde\vr_\e\log \tilde\vr_\e) \tilde\vu_\e\big)+\tilde\vr_\e \dive \tilde\vu_\e=0,\quad \dive\big((\vr\log \vr) \vu\big)+\vr \dive \vu=0,\ \mbox{in}\ \mathcal{D}'(\O).
\ee
Passing $\e\to 0$ in the first equation of \eqref{eq:slogs} gives
\be\label{eq:slogs1}
\dive\big(\overline{(\vr\log \vr)}\, \vu\big)+\overline{\vr \dive \vu}=0, \ \mbox{in}\ \mathcal{D}'(\O),
\ee
where we used the strong convergence of the velocity in \eqref{st-con-vue} and
\ba\label{eq:slogs2}
&\tilde\vr_\e\log \tilde\vr_\e \to \overline{\vr\log \vr} \quad \mbox{weakly in}\quad L^{q}(\O)\ \mbox{for any $q<3\g-3$},\\
&\tilde\vr_\e \dive \tilde\vu_\e \to \overline{\vr \dive \vu} \quad \mbox{weakly in}\quad L^{\frac{6\g-6}{3\g-1}}(\O).
\ea

Then for any $\psi\in C_c^\infty(\Omega)$, \eqref{eq:slogs1} and \eqref{eq:slogs2} implies
\be\label{flux1}
\lim_{\e\to 0}\intO{\psi \left(p(\tilde\vr_\e)-(\frac{4\mu}{3}+\eta)\dive \tilde\vu_\e\right)\tilde\vr_\e}=\intO{\psi \overline {p(\vr)\vr} -(\frac{4\mu}{3}+\eta) \overline{(\vr\log \vr)}\, \vu \cdot \nabla \psi}.
\ee

By the second equation in \eqref{eq:slogs}, we obtain
\be\label{flux2}
\intO{\psi\left(\overline{p(\vr)}-(\frac{4\mu}{3}+\eta)\dive \vu\right)\vr}=\intO{\psi \overline{p(\vr)}\vr -(\frac{4\mu}{3}+\eta) (\vr\log \vr)\, \vu \cdot \nabla \psi}.
\ee

Let $\{\psi_n\}_{n\in \Z_+} \subset C_c^\infty(\Omega)$ such that
$$
 \psi_n(x)=0 \ \mbox{if} \ d(x,\partial \Omega)<\frac{1}{n},\quad \psi_n(x)=1 \ \mbox{if} \ d(x,\partial \O)>\frac{2}{n},\quad \|\nabla  \psi_n\|_{L^\infty(\O; R^3)}\leq 4n.
$$
Then for any $q\in [1,\infty]$:
\[
\|1-\psi_n\|_{L^q(\O)}\leq C\, n^{-\frac{1}{q}},\quad \|\nabla \psi_n\|_{L^q(\O; R^3)}\leq C\, n^{1-\frac{1}{q}},
\]
and consequently
\[
\|d(x,\d \O)\nabla \psi_n\|_{L^q; R^3}\leq C\, n^{-\frac{1}{q}}.
\]

The fact $\vu\in W_0^{1,2}(\O; R^3)$ implies
\[
[d(x,\partial \O)]^{-1} \vu \in L^2(\O; R^3).
\]
Therefore,
\ba\label{est-slogs1}
&\intO{\nabla \psi_n \cdot \overline{(\vr\log\vr)}\vu} \\
&\leq \|d(x,\partial \O)\nabla \psi_n\|_{L^{10}(\O; R^3)}\|\overline{(\vr\log\vr)}\|_{L^{5/2}(\O)}\|[d(x,\partial \Omega)]^{-1}\vu\|_{L^2(\O; R^3)}\\
&\leq C\, \ n^{-1/10}.
\ea
Similarly,
\be\label{est-slogs2}
\intO{\nabla \psi_n \cdot (\vr\log\vr)\vu} \leq C\, \ n^{-1/10}.
\ee

We choose $\psi=\psi_n$ in \eqref{flux} and pass to the limit $n\to \infty$. By using \eqref{flux1}, \eqref{flux2}, \eqref{est-slogs1} and \eqref{est-slogs2}, we deduce
\be\label{flux-f}
\intO{\overline{p(\vr)\vr}-\overline{p(\vr)}\vr}=0.
\ee

By the strict monotonicity of the mapping $\vr \mapsto p(\vr)$, applying Theorem 10.19 in \cite{F-N-book} or Lemma 3.35 in \cite{N-book} implies
\[
\overline{p(\vr)\vr} \geq \overline{p(\vr)}\vr,\quad \mbox{a.e. in}\quad \O.
\]
Together with \eqref{flux-f}, we deduce
\[
\overline{p(\vr)\vr} = \overline{p(\vr)}\vr,\quad \mbox{a.e. in}\quad \O.
\]

We have completed the proof of Lemma \ref{lem:flux2}.

\end{proof}

Thanks to the monotonicity of $p(\cdot)$, again by Theorem 10.19 in \cite{F-N-book}, we obtain $\overline{p(\vr)}=p(\vr)$. Hence, we complete the proof of Theorem \ref{thm-hom}.

For the convenience of readers, we recall Theorem 10.19 in \cite{F-N-book}: Let $I\subset R$ be an interval, $Q\subset R^d$ be a domain, $P$ and $G$ be non-decreasing functions in $C(I)$. Let $\{\vr_n\}_{n\in \N}$ be a sequence in $L^1(Q;I)$ such that
$$
P(\vr_n) \to \overline{P(\vr)}, \quad G(\vr_n) \to \overline{G(\vr)}, \quad P(\vr_n)G(\vr_n) \to \overline{P(\vr)G(\vr)}, \quad \mbox{weakly in $L^1(Q)$}.
$$
Then the following properties hold:
\begin{itemize}

\item[(i).] $\overline{P(\vr)}\ \overline{G(\vr)}\leq \overline{P(\vr)G(\vr)}.$
\item[(ii).] If, in addition, $P\in C(R), \ G\in C(R),\ G(R)=R$, $G$ is strictly increasing, and $\overline{P(\vr)}\ \overline{G(\vr)} = \overline{P(\vr)G(\vr)}$, then $\overline{P(\vr)}=P\circ G^{-1} \overline{G(\vr)} $.  If, in particular, $G(z)=z$ be the identity function, there holds $\overline{P(\vr)}=P(\vr)$.

\end{itemize}

\section{Conclusions and perspectives}

In this paper, we constructed an inverse of the divergence operator in a domain perforated with tiny holes and we showed the precise and optimal dependency on the size of the holes for the norm of this inverse operator; in particular, under some smallness constrain, this inverse of the divergence operator is uniformly bounded.  We apply such an operator in the study of homogenization problems for stationary compressible Navier-Stokes system. Under some constrain (see \eqref{res-g-a}) between the adiabatic exponent and the size of the holes, we show that the homogenization process does not change the motion of the fluids: in the limit, we obtain again compressible Navier-Stokes equations.

Here we focus on the case where the holes are very small, corresponding to $\a>3$. It is also known that if $\a=1$, one can recover Darcy's law from the homogenization. However, the case with $1<\a\leq 3$, in particular the critical case $\a=3$ is still open.

\section*{Acknowledgements}
\thispagestyle{empty}

Eduard Feireisl and Yong Lu acknowledges the support of the project LL1202 in the programme ERC-CZ funded by the Ministry of Education, Youth and Sports of the Czech Republic.


\end{document}